\documentclass[a4paper,11pt]{amsart}

\usepackage{amssymb,amsmath,amsthm,amscd}
\usepackage[ansinew]{inputenc} 
\usepackage[english]{babel}

\usepackage{newtxtext}
\usepackage{newtxmath}
\usepackage[T1]{fontenc}
\usepackage{mathrsfs}
\usepackage{amscd}
\usepackage{currfile}
\usepackage{hyperref}
\usepackage{bm}
\usepackage{currfile}

\newtheorem{theorem}{Theorem}
\newtheorem*{theorem*}{Theorem}

\newtheorem{lemma}[theorem]{Lemma}
\newtheorem{proposition}[theorem]{Proposition}
\newtheorem{corollary}[theorem]{Corollary}
\newtheorem{example}[theorem]{Example}

\theoremstyle{definition}
\newtheorem{remark}[theorem]{Remark}

\newcommand{\hh}{{\mathbb{H}}}
\newcommand{\HH}{{\mathbb{H}}}

\newcommand{\cc}{{\mathbb{C}}}
\newcommand{\rr}{{\mathbb{R}}}
\newcommand{\zz}{{\mathbb{Z}}}
\newcommand{\nn}{{\mathbb{N}}}

\newcommand{\s}{{\mathbb{S}}}

\newcommand{\SR}{\mathcal{SR}}
\newcommand{\SF}{\mathcal{S}}
\newcommand{\ZH}{\mathcal{ZH}}
\newcommand{\AM}{\mathcal{AM}}

\newcommand{\I}{\mathcal{I}}
\newcommand{\LL}{\mathcal{D}}
\newcommand{\Pn}{\mathcal{P}}
\newcommand{\dcf}{\dibar}
\newcommand{\dibar}{\overline\partial}
\newcommand{\cdcf}{\partial}
\newcommand{\dif}{\vartheta}
\newcommand{\difbar}{\overline{\dif}}
\newcommand\IM{\operatorname{Im}}

\newcommand\vs[1]{{#1}_s^\circ}
\newcommand\sd[1]{{#1}'_s}

\newcommand\im{\operatorname{Im}}
\newcommand\Log{\operatorname{Log}}

\newcommand\sinc{\operatorname{sinc}}

\newcommand{\ui}{\imath}

\newcommand{\OO}{\Omega}

\newcommand{\mbb}{\mathbb}

\newcommand{\R}{\mbb{R}}

\newcommand{\dd[3]}{\frac{\partial #2}{\partial #3}}
\newcommand{\C}{\mbb{C}}

\newcommand{\Zt}{\widetilde{\mathcal Z}}

\newcommand{\ID}{\widetilde\Delta}

\begin{document}

\title{A local Cauchy integral formula for slice-regular functions}

\author{Alessandro Perotti}
\address{Department of Mathematics, University of Trento, Via Sommarive 14, Trento Italy}
\email{alessandro.perotti@unitn.it}

\thanks{The author is a member the INdAM Research group GNSAGA and was supported by the grants ``Progetto di Ricerca INdAM, Teoria delle funzioni ipercomplesse e applicazioni'', and PRIN ``Real and Complex Manifolds: Topology, Geometry and holomorphic dynamics''.}

\begin{abstract}
We prove a Cauchy-type integral formula for slice-regular functions where the integration is performed on the boundary of an open subset of the quaternionic space, with no requirement of axial symmetry. In particular, we get a local Cauchy-type integral formula. 
As a step towards the proof, we provide a decomposition of a slice-regular function as a combination of two axially monogenic functions. 
\end{abstract}


\maketitle


\section{Introduction}
Cauchy's integral formula is one of the most powerful tools in complex analysis. It plays a key role also in the study of 
any function theory that aims to extend complex analysis to higher dimensional algebras. In the four dimensional case represented by the quaternionic skew field, there are at least two different generalizations of the concept of holomorphic functions. The first one deals with functions in the kernel of the Cauchy-Riemann-Fueter differential operator $\dcf := \frac{1}{2}\Big(\frac{\partial }{\partial x_0} + i \frac{\partial }{\partial x_1}  + j \frac{\partial }{\partial x_2}  + k \frac{\partial }{\partial x_3} \Big)$, where $i,j,k$ are the quaternionic basic imaginary units and $x=x_0+ix_1+jx_2+kx_3$ is the real representation of a quaternion $x$. These functions, usually called Fueter-regular or monogenic, have been studied extensively for many  decades. Primary references are the article of Fueter \cite{Fueter1934}, where the Cauchy's integral theorem was proved, the 
paper of Sudbery \cite{Su} and the monograph \cite{BDS}, where the results were given in their generality in the context of Clifford analysis. 
The second function theory, introduced in 2006-2007 by Gentili and Struppa \cite{GeSt2006CR,GeSt2007Adv} following an idea of Cullen \cite{Cullen}, with the objective to include quaternionic polynomials and series, is the theory of quaternionic slice-regular functions. This function theory is based on the particular complex-slice structure of the quaternionic space $\mathbb{H}$. In Section~2 we briefly recall 
the definitions and properties of slice functions and slice-regular functions that are used in the subsequent sections. 

A Cauchy-type integral formula with slice-regular kernel was proved in 
\cite{CoGeSaAGAG}. In this formula integration is performed over the boundary of a two-dimensional domain having an axial symmetry with respect to the real axis. A volume Cauchy-type formula, where integration is made over the boundary of an open axially symmetric domain, was proved in \cite{VolumeCauchy} in the more general context of real alternative *-algebras. That result extended to every slice-regular function a similar one obtained by Cullen \cite{Cullen} on the quaternions. 

The aim of the present paper is to prove a Cauchy-type integral formula (Theorem \ref{thm:Cauchy}) for slice-regular functions where the integration is performed on the boundary of a not necessarily axially symmetric open set. As a corollary, we obtain a local Cauchy-type integral formula for slice-regular functions (Corollary \ref{cor:localCauchy}). The integral kernel is not slice-regular, but it is universal, i.e., not depending on the domain. 
An unavoidable aspect of the formula is the appearing, along with the boundary values of the slice-regular function, of the values of a complementary 
function, namely, the slice derivative of a slice primitive of the function.

The Cauchy-type formula is proved using facts from both the above-mentioned quaternionic functions theories. We show (Theorem \ref{teo:AMDec}) that every slice-regular function can be expressed as a combination of two axially monogenic functions, to which the Fueter's version of Cauchy's integral formula apply. We recall that an axially monogenic function is a monogenic slice function, i.e., a slice function in the kernel of the operator $\dcf$. 
We also give a new proof (Theorem \ref{teo:Surjectivity}) of the surjectivity of the Laplacian mapping from the space of slice-regular functions to that of axially monogenic functions (see \cite{TheInverseFueter}). This result has a role in proving the uniqueness of the above-mentioned decomposition in terms of axially monogenic functions.

\section{Preliminaries}\label{sec:pre}

The slice function theory of one quaternionic variable \cite{GeSt2006CR,GeSt2007Adv} 
is based on the slice decomposition of the quaternionic space $\hh$. For each element $J$ in the sphere of quaternionic imaginary units
 \[\s=\{J\in\HH\ |\ J^2=-1\}=\{x_1i+x_2j+x_3k\in\HH\ |\ x_1^2+x_2^2+x_3^2=1\},\]
we let $\C_J=\langle 1,J\rangle\simeq\C$ be the subalgebra generated by $J$. Then it holds
\[\HH=\bigcup_{J\in \s}\C_J, \quad\text{with $\C_J\cap\C_K=\R$\ for every $J,K\in\s,\ J\ne\pm K$.}\]
A differentiable function $f:\OO\subseteq\HH\rightarrow\HH$  is called \emph{(left) slice-regular} \cite{GeSt2007Adv} on the open set $\OO$ if, for each $J\in\s$, the restriction
$f_{\,|\OO\cap\C_J}\, : \, \OO\cap\C_J\rightarrow \HH$
is holomorphic with respect to the complex structure defined by left multiplication by $J$. 
We refer the reader to \cite{GeStoSt2013} 
and the references therein for more results in this function theory. 

Another approach to slice regularity was introduced in \cite{GhPe_Trends,GhPe_AIM} (see also \cite{AlgebraSliceFunctions}
 for recent developments), making use of the concept of \emph{slice functions}. 
We briefly recall their definition and some operations on them.
Given a set $D\subseteq\C$, invariant with respect to complex conjugation, a function $F: D\to \hh\otimes\C$ that satisfies $F(\overline z)=\overline{F(z)}$ for every $z\in D$ (the conjugation in $\hh\otimes\C$ is induced by complex conjugation in the second factor) is called a \emph{stem function} on $D$, a concept already present in seminal works of Fueter \cite{Fueter1934} and Cullen \cite{Cullen}.

Let  $\Phi_J:\C\to\C_J$ be the canonical isomorphism that maps $a+\ui b\in\cc$ to $a+Jb$ (with $\ui^2=-1$). Given an open set $D\subseteq\cc$, let $\OO_D=\cup_{J\in\s}\Phi_J(D)\subset\hh$. 
Open sets in $\hh$ of the form $\OO=\OO_D$ are called \emph{axially symmetric} sets. 
An axially symmetric connected set $\OO$ is called a \emph{slice domain} if $\OO\cap \R\ne\emptyset$, a \emph{product domain} if $\OO\cap \R=\emptyset$. Any axially symmetric open set is union of a family of domains of these two types. 

The stem function $F=F_1+\ui F_2$  on $D$ (with $F_1,F_2:D\to\hh$) induces the \emph{slice function} $f=\I(F):\OO_D \to \hh$ as follows:
 if $x=\alpha+J\beta=\Phi_J(z)\in \OO_D\cap \C_J$, then
\[f(x)=F_1(z)+JF_2(z).\]
The slice function $f$ is called \emph{slice-regular} if $F$ is holomorphic w.r.t.\ the complex structure induced on $\hh\otimes\C$ by the second factor. 
If a domain $\OO$ in $\hh$ is axially symmetric and intersects the real axis, then this definition of slice regularity is equivalent to the one proposed by Gentili and Struppa \cite{GeSt2007Adv}. We will denote by $\SR(\OO)$ the right quaternionic module of slice-regular functions on $\OO$ and by $\SF^1(\OO)$ the class of slice functions induced by stem functions of the class $C^1$ on $\OO$.


The \emph{slice product} of two slice functions $f=\I(F)$, $g=\I(G)$ on $\OO=\OO_D$ is defined by means of the pointwise product of the stem functions:
\[f\cdot g=\I(FG).\]
The function $f=\I(F)$ is called \emph{slice-preserving} if the $\hh$-components $F_1$ and $F_2$ of the stem function $F$ are real-valued.
This is equivalent to the condition $f(\overline x)=\overline{f(x)}$ for every $x\in\OO$.
If $f$ is slice-preserving, then $f\cdot g$ coincides with the pointwise product of $f$ and $g$. 
If $f,g$ are slice-regular on $\OO$, then also their slice product $f\cdot g$ is slice-regular on $\OO$. 


The \emph{slice derivatives} $\dd{f}{x},\dd{f\;}{x^c}$ of a slice functions $f=\I(F)$ are defined by means of the  Cauchy-Riemann operators applied to the inducing stem function $F$:
\[\dd{f}{x}=\I\left(\dd{F}{z}\right),\quad \dd{f\;}{x^c}=\I\left(\dd{F}{\overline z}\right).\]
It follows that $f$ is slice-regular if and only if $\dd{f\;}{x^c}=0$ and if $f$ is slice-regular on $\OO$ then also $\dd{f}{x}$ is slice-regular on $\OO$. Moreover, the slice derivatives satisfy the Leibniz product rule w.r.t.\ the slice product. If $f=\dd g x$, we will say that $g$ is a \emph{slice primitive} of $f$.

We recall other two useful concepts introduced in \cite{GhPe_AIM}. Given a slice function $f=\I(F_1+\ui F_2)$ on $\OO$, the function $\vs f:\OO \to \hh$, called \emph{spherical value} of $f$, and the function $f'_s:\OO \setminus \R \to \hh$, called  \emph{spherical derivative} of $f$, are defined as
\[
\vs f(x):=\tfrac{1}{2}(f(x)+f(\overline x))
\quad \text{and} \quad
f'_s(x):=\tfrac{1}{2} \im(x)^{-1}(f(x)-f(\overline x)).
\] 
The functions $\vs f=\I(F_1)$ and $f'_s=\I(\beta^{-1}F_2)$ are slice functions, constant on 2-\emph{spheres}\, $\s_x=\alpha+\s\beta$ for any $x=\alpha+J\beta\in\OO\setminus\R$, and such that 
\begin{equation}\label{eq:spherical}
  f(x)=\vs{f}(x)+\im(x)f'_s(x)
\end{equation}
for every $x\in\OO\setminus \R$. 
The spherical derivative satisfies a Leibniz-type rule w.r.t.\ the slice product (see \cite[\S5]{GhPe_AIM}): $\sd{(f\cdot g)}=\sd f\vs g+\vs f\sd g$.
For any slice-regular function $f$ on $\OO$, $\sd{f}$ extends as the slice derivative $\dd{f}{x}$ on $\OO\cap\R$.
Note that if  $f$ is slice-regular and $\sd f\equiv0$, then $F_2\equiv0$ and $F$ is locally constant. Another remarkable property of the spherical derivative of a slice-regular function is its harmonicity when considered as a map of four real variables (see Theorem \ref{thm:harmonicity} of Section 3). 

\section{An axially monogenic decomposition for slice-regular functions}

Let $\dcf$ denote the Cauchy-Riemann-Fueter operator
\[\dcf =\frac12\left(\dd{}{x_0}+i\dd{}{x_1}+j\dd{}{x_2}+k\dd{}{x_3}\right).\]
Given an axially symmetric domain $\OO$ of $\hh$, let $\AM(\OO)$ be the class of axially monogenic functions, i.e., of monogenic slice functions on $\OO$:
\[\AM(\OO)=\{f\in\SF^1(\OO)\,|\,\dcf f=0\}.
\]
There is a result, usually called Fueter's Theorem \cite{Fueter1934}, which in its generalised form can be seen as a bridge between the class of slice-regular functions and the one of monogenic functions. 
We report this result from \cite[Prop.~3.61, Cor.~3.6.2 and Thm.~3.6.3]{Harmonicity}, where some formulas linking the spherical derivative of slice functions with the Cauchy-Riemann-Fueter operator were proved. 
Let $\Delta$ denote the Laplacian operator in $\rr^4$. 
In the statement of this result we will also need the global differential operator $\difbar$ introduced in \cite{Gh_Pe_GlobDiff}. For slice functions $f\in C^1(\OO)$, it holds $\difbar f=\dd{f}{x^c}$ on $\OO\setminus\rr$ (see \cite[Theorem 2.2]{Gh_Pe_GlobDiff}).

\begin{theorem}[\cite{Harmonicity}]\label{thm:harmonicity}
Let $\OO$ be an axially symmetric domain 
in $\HH$. Let  $f:\OO\to\HH$ be a slice function of class $\mathcal{C}^1(\OO)$. Then
\begin{enumerate}
\item
$\dcf f=\difbar f-\sd f=\dd{f\ }{x^c}-\sd f$ on $\OO\setminus\rr$.
In particular, 
$f$ is slice-regular if and only if $\,\dcf f=-f'_s$, and $f$ is axially monogenic if and only if $\dd f{x^c}=\sd f$.
\item
If $f:\OO\to\HH$ is slice-regular, then it holds:
\begin{enumerate}
\item
  The four real components of $f'_s$ are harmonic on $\OO$. 
\item
The following generalization of Fueter's Theorem holds:
\[\dcf\Delta f=\Delta \dcf f=-\Delta f'_s=0.\]
\item
$\Delta f=-4\,\dd{\sd f}{x}$.
\end{enumerate}
\end{enumerate}
\end{theorem}

In the following we will use also the following result from \cite{EigenvaluePbms}, which shows that the Laplacian of a slice-regular function can be expressed by first order derivatives.

\begin{lemma}[\cite{EigenvaluePbms} Lemma 23]\label{lem:Laplacian} 
If $f\in\SR(\OO)$, then for every $x\in\OO$ it holds
\[\label{eq:Deltaf1}
\IM(x)\Delta f=2\left(\sd f-\dd{f}{x}\right)=-2\left(\dcf f+\dd{f}{x}\right).
\]
\end{lemma}

It is known that the Laplacian (also called Fueter mapping in this context) maps the space $\SR(\OO)$ onto $\AM(\OO)$ surjectively (see \cite{TheInverseFueter} and \cite{DongQian} and references therein). 
If $\OO$ is connected, the inverse image in $\SR(\OO)$ of a function $g\in\AM(\OO)$ under $\Delta$  is unique up to a quaternionic affine function $xa+b$ (see \cite[Lemma 23(a)]{EigenvaluePbms}).
Using Lemma \ref{lem:Laplacian}, we now give an elementary proof of the surjectivity of $\Delta$ under suitable topological hypotheses.

\begin{theorem}
\label{teo:Surjectivity}
Let $\OO=\OO_D$ be an axially symmetric open set in $\hh$. Assume that every connected component of $D$ is simply connected. Then the Laplacian 
\[
\Delta:\SR(\OO)\to\AM(\OO)
\]
is surjective. 
\end{theorem}
\begin{proof}
Let $g=\I(G)=\I(G_1+\ui G_2)\in\AM(\OO)$. Let $\{e_0,e_1,e_2,e_3\}$ be a real basis of $\hh$. The decomposition $G=\sum_{i=0}^3G^ie_i$ defines four stem functions $G^i:D\to\rr\otimes\cc\simeq\cc$. Let $G^i=G^i_1+\ui G^i_2$, with $G^i_1,G^i_2$ real valued. 
Since $g\in\AM(\OO)$, in view of Theorem \ref{thm:harmonicity} it holds $\I(\dd G{\overline z})=\dd g{x^c}=\sd g=\I(\beta^{-1}{G_2})$, i.e., $\dd G{\overline z}=\beta^{-1}{G_2}$. Then
\[
\dd G{\overline z}=\sum_{i=0}^3\dd{G^i}{\overline z}e_i=\beta^{-1}{G_2}=\sum_{i=0}^3(\beta^{-1}{G^i_2})e_i,
\]
which implies $\dd{G^i}{\overline z}=\beta^{-1}{G^i_2}$ for each $i=0,\ldots,3$. Therefore $g=\sum_{i=0}^3g^ie_i$, with every $g^i=\I(G^i)\in\AM(\OO)$ and slice-preserving. If we find $f^i\in\SR(\OO)$ such that $\Delta(f^i)=g^i$, then $\Delta(\sum_{i=0}^3f^ie_i)=g$. We can then assume that $g\in\AM(\OO)$ is slice-preserving, i.e., $G_1$ and $G_2$ are real valued. 

Our aim is to find a slice-preserving $f\in\SR(\OO)$ such that $\IM(x)g=2(\sd f-\dd f x)$, since then Lemma \ref{lem:Laplacian} gives $\Delta f=g$. 
As above, the condition $\dcf g=0$ is equivalent to $\dd G{\overline z}=\beta^{-1}G_2$, i.e., if $z=\alpha+\ui\beta$, 
\begin{equation}\label{eq:dg}
\dd{G_1}{\alpha}-\dd{G_2}{\beta}=2\beta^{-1}G_2,\quad \dd {G_1}{\beta}+\dd{G_2}{\alpha}=0.
\end{equation}
Since any axially symmetric open set is union of a family of slice domains or product domains, we can assume that $\OO$ is a domain of one of these types.

If $\OO=\OO_D$ is a slice domain, $D$ is a simply connected subset of $\cc$. From the second equality in \eqref{eq:dg}, we can find $H\in C^\infty(D,\rr)$ such that $4\dd H{\overline z}=2\left(\dd{H}{\alpha}+\ui\dd{H}{\beta}\right)=-\overline G=-G_1+\ui G_2$. Let 
\[
F_2(z):=\frac{\beta}2(H(z)+H(\overline z))\in C^\infty(D,\rr).
\]
Then $F_2$ is odd w.r.t.\ $\beta$, and it holds $4\dd{}{\overline z}\left(\frac{F_2}{\beta}\right)=-\overline G$. A direct computation shows that $\Delta F_2=G_2-\frac{\beta}2\left(\dd{G_1}{\alpha}-\dd{G_2}{\beta}\right)$. The first equality in \eqref{eq:dg} then implies that $F_2$ is harmonic on $D$. Let $F_1\in C^\infty(D,\rr)$ be an harmonic conjugate of $F_2$. Replacing $F_1$ with $(F_1(z)+F_1(\overline z))/2$ if necessary, we get that $F_1$ is even w.r.t.\ $\beta$ and that $F:=F_1+\ui F_2$ is holomorphic on $D$, i.e., $F$ is a stem function on $D$, inducing a slice-regular $f:=\I(F)$ on $\OO$. It remains to show that $2(\sd f-\dd f x)=\IM(x)g$, or, equivalently, that 
\[
4\left(\frac{F_2}{\beta}-\dd F z\right)=(z-\overline z)G.
\]
It holds 
\[
-\overline G=4\dd{}{\overline z}\left(\frac{F_2}{\beta}\right)=\frac4{\beta}\dd{F_2}{\overline z}+\frac{2F_2}{\ui\beta^2}\quad \Rightarrow\quad
-G=\frac4{\beta}\dd{F_2}{z}-\frac{2F_2}{\ui\beta^2}.
\]
Then
\[
4\left(\frac{F_2}{\beta}-\dd F z\right)=4\frac{F_2}{\beta}-2\ui\left(4\dd {F_2} z\right)=4\frac{F_2}{\beta}-2\ui\left(-\beta G+\frac{2F_2}{\ui\beta}\right)= 2\ui\beta G.
\]

If $\OO=\OO_D$ is a product domain, let $D^+=D\cap\cc^+$ and $D^-=D\cap\cc^-$. Since $D^+$ is simply connected, there exist $H\in C^\infty(D^+,\rr)$ such that $4\dd H{\overline z}=-\overline G$ on $D^+$. $H$ can be extended to $D$ setting $H(z)=H(\overline z)$ for $z\in D^-$. Then as in the previous case $F_2:=\beta H$ is harmonic on $D$, odd w.r.t.\ $\beta$. Let $F_1\in C^\infty(D^+,\rr)$ such that $F=F_1+\ui F_2$ is holomorphic on $D^+$. Setting $F_1(z)=F_1(\overline z)$ for $z\in D^-$, $F$ can be extended to a holomorphic stem function on $D$ satisfying as above $4\left(\frac{F_2}{\beta}-\dd F z\right)=(z-\overline z)G$ on $D$.   
\end{proof}

In view of Theorem \ref{thm:harmonicity}(2c), a right inverse $\ID:\AM(\OO)\to\SR(\OO)$ of $\Delta$ can be defined on axially monogenic polynomials as in \cite[Prop.~24]{EigenvaluePbms}. It associates for every $n\in\zz$ the slice-regular monomial $-\frac14x^{n+2}$ to the rational function $\Pn_n$, defined by
\begin{equation}\label{eq:defPn}
\Pn_n(x):=-\tfrac14\Delta(x^{n+2})=
\dd{}{x}((\sd{x^{n+2})}).
\end{equation}
The functions $\Pn_n$ are axially monogenic and then harmonic. They are slice-preserving functions (not slice-regular for $n\ne0$) on $\hh$. They were computed already by Fueter in \cite{Fueter1934} (see formula (12) on p.\ 316) and afterwords used by many authors. 
For $n\ge0$ the functions $\Pn_n$ are polynomials of degree $n$ in $x_0$, $x_1$, $x_2$, $x_3$. For $n<0$ they are homogeneous functions on $\hh\setminus\{0\}$, still of degree $n$.  
The functions $\Pn_n$ and  $\Pn_{-n}$ are related through the Kelvin transform of $\rr^4$ (\cite[Prop.6.7(c)]{Harmonicity}). 
In particular, $\Pn_{-1}=\Pn_{-2}\equiv0$, while $\Pn_{-3}(x)=\overline x/|x|^4$ is equal, up to a multiplicative constant, to the Cauchy-Fueter kernel $E(x)=\frac1{2\pi^2}\frac{\overline x}{|x|^4}$ (see \cite{Fueter1934} and also \cite{Su}, \cite[Ch.3]{GHS} for more recent expositions).

For $n\ge0$, the polynomials $\Pn_n$ are related to the spherical derivatives $\Zt_{n}(x):=\sd{(x^{n+1})}$ of quaternionic powers. These functions are harmonic homogeneous polynomials of degree $n$ in the four real variables $x_0$, $x_1$, $x_2$, $x_3$. The polynomials $\Zt_{n}$ are called \emph{zonal harmonic polynomials with pole 1}, since they have an axial symmetry with respect to the real axis (see \cite[Ch.5]{HFT} and \cite{Harmonicity,AlmansiH}).




We are now able to write a decomposition of quaternionic polynomials in terms of a pair of axially monogenic polynomials.

\begin{proposition}\label{pro:powers}
For every $n\in\nn$, it holds
\begin{equation}\label{eq:powers}
(n+1)x^n=\Pn_n(x)-\overline x\,\Pn_{n-1}(x)
\end{equation}
for every $x\in\hh$. 
\end{proposition}
\begin{proof}
It was proved in \cite[Cor.\ 6.7]{Harmonicity} that for every $n\in\nn$, it holds
\begin{equation*}
x^{n+1}=\Zt_{n+1}(x)-\overline x\, \Zt_{n}(x)=\Zt_{n+1}(x)-\overline x\cdot \Zt_{n}(x)\text{\quad  $\forall x\in\hh$}.
\end{equation*}
Here the pointwise and slice products coincide since $\overline x$ is a slice-preserving function. Taking slice derivatives and using Leibniz property, we get
\begin{equation*}
(n+1)x^{n}=\dd{\Zt_{n+1}(x)}{x}-\overline x\cdot \dd{\Zt_{n}(x)}{x}=\Pn_n(x)-\overline x\cdot\Pn_{n-1}(x)=\Pn_n(x)-\overline x\,\Pn_{n-1}(x).
\end{equation*}
\end{proof}

\begin{corollary}
\label{teo:AMDecPol}
Let $P\in\hh[X]$ have degree $d\ge1$. There exist two axially monogenic polynomials $Q_1$, $Q_2$, of degrees $d$ and $d-1$ respectively, such that 
\begin{equation*}\label{eq:AMDecPol}
P(x)=Q_1(x)-\overline x Q_2(x)\text{\quad  $\forall x\in\hh$.}
\end{equation*}
\end{corollary}
\begin{proof}
Let $P(x)=\sum_{n=0}^dx^na_n$. The thesis follows immediately from Proposition \ref{pro:powers} by setting
\[
Q_1(x)=\sum_{n=0}^d\frac{\Pn_n(x)}{n+1}a_n\text{\quad and\quad}Q_2(x)=\sum_{n=1}^d\frac{\Pn_{n-1}(x)}{n+1}a_n.
\]
\end{proof}

Corollary \ref{teo:AMDecPol} can be generalized to every slice-regular function. Before doing it, we show one more general property of axially monogenic functions. 

\begin{lemma}\label{lem:am}
Let $\OO$ be as in Theorem \ref{teo:Surjectivity}.
If both $g$ and $\overline x g$ are axially monogenic on $\OO$, then $g$ is identically zero.
\end{lemma}
\begin{proof}
Let $\LL$ be the first order linear operator defined for any slice function $h$ of class $C^2(\OO)$ by 
\begin{equation}\label{eq:D}
\LL(h)=\dcf\left(\overline x\dd{h}{x}\right).
\end{equation}
We claim that if $f\in\SR(\OO)$, then $\LL(\sd f)=\sd{\left(\dd{f}{x}\right)}$. Indeed, it holds
\begin{align*}
\overline x\dd{\sd f}{x}&=\dd{(\overline x\sd f)}{x}=\dd{(x_0\sd f-\IM(x)\sd f)}{x}=\dd{(x_0\sd f+\vs f)}{x}+\dd{(-\vs f-\IM(x)\sd f)}{x}\\
&=\dd{(\sd{(xf)})}{x}-\dd{f}{x}=-\frac14\Delta(xf)-\dd{f}{x},
\end{align*}
where we used Theorem \ref{thm:harmonicity}(2c). Then 
\[
\LL(\sd f)=-\frac14\dcf\left(\Delta(xf)\right)-\dcf\left(\dd{f}{x}\right)=-\dcf\left(\dd{f}{x}\right)=\sd{\left(\dd{f}{x}\right)}
\]
thanks to point (2b) of Theorem \ref{thm:harmonicity} applied to the slice-regular function $xf$ and point (1) of the same theorem applied to $\dd f x$.

If $g\in\AM(\OO)$, thanks to the surjectivity of $\Delta:\SR(\OO)\to\AM(\OO)$ we can assume that $g=\Delta f$, with $f\in\SR(\OO)$. If also $\overline x g\in\AM(\OO)$, then $0=\dcf(\overline x g)=\dcf(\overline x\Delta f)=-4\dcf\left(\overline x\dd{\sd f}{x}\right)=-4\LL(\sd f)=-4\sd{\left(\dd f x\right)}$. Therefore $\dd f x$ is locally constant,  $f$ is (locally) an affine function $f(x)=xa+b$, with $a,b\in\hh$ and $g=\Delta f=0$.
\end{proof}

\begin{remark}
Let $\LL$ be the operator defined in \eqref{eq:D}. The claim given in the proof of Lemma \ref{lem:am} shows that
\[
\LL(\Zt_{n+1})=(n+2)\Zt_{n},\text{\quad i.e.,\quad}\dcf(\overline x\Pn_n)=(n+2)\Zt_n
\]
for every $n\in\nn$.
\end{remark}

Now we extend Corollary \ref{teo:AMDecPol} to every slice-regular function. Since every monogenic function is harmonic, the result we obtain can be seen as a refinement of the Almansi type decomposition proved in \cite[Theorem 4]{AlmansiH}

\begin{theorem}
\label{teo:AMDec}
Let $f$ be slice-regular on an axially symmetric open set $\OO$. Assume that $\OO=\OO_D$, and that every connected component of $D$ is simply connected.
Then there exist two uniquely determined axially monogenic functions $g_1$ and $g_2$, such that 
\[
f(x)=g_1(x)-\overline x g_2(x)\text{\quad  $\forall x\in\OO$.}
\]
The functions $g_1$ and $g_2$ can be computed from a slice-regular primitive of $f$. If $f=\dd{g}{x}$ on $\OO$, with $g\in\SR(\OO)$, then
\[
g_1= -\tfrac14\Delta(xg),\quad  g_2=-\tfrac14\Delta g.
\]
Moreover, $f$ is slice-preserving if and only if $g_1$ and $g_2$ are slice-preserving.
\end{theorem}
\begin{proof}
Since any axially symmetric open set is union of a family of slice domains or product domains, we can assume that $\OO$ is a domain of one of these types.
Assume that there exists $g\in\SR(\OO)$ such that $f=\dd{g}{x}$ on $\OO$. Then, using Theorem \ref{thm:harmonicity}(2c), the Leibniz-type formula for spherical derivative (see \cite[\S5]{GhPe_AIM}) and \eqref{eq:spherical} we get
\begin{align*}
-\tfrac14\left(\Delta(xg)-\overline x\Delta g\right)&=\dd{(\sd{(xg)})}{x}-\overline x \dd{(\sd{g})}{x}=\dd{}{x}\left(x_0\sd g+\vs g\right)-\overline x\dd{(\sd{g})}{x}\\
&=\dd{}{x}\left(x_0\sd g+\vs g-\overline x\sd g\right)=\dd{}{x}\left(\vs g+\IM(x)\sd g\right)=\dd{g}{x}=f.
\end{align*}
The functions $g_1:= -\tfrac14\Delta(xg)$ and $g_2=-\tfrac14\Delta g$ are axially monogenic on $\OO$ thanks to Fueter's Theorem \ref{thm:harmonicity}(2b).
To conclude the existence part of the proof it remains to show that there exists a slice-regular primitive of $f$ on $\OO$. Let $\{e_0,e_1,e_2,e_3\}$ be a real basis of $\hh$. If $f=\I(F)$, the decomposition $F=\sum_{i=0}^3F^ie_i$ defines four holomorphic stem functions $F^i:D\to\rr\otimes\cc\simeq\cc$.  

If $\OO=\OO_D$ is a slice domain, by assumption $D$ is a simply connected subset of $\cc$. Let $G^i:D\to\cc$ be a holomorphic primitive of $F^i$, for $i=0,1,2,3$ and let $\tilde G^i$ be defined on $D$ by
\[
\tilde G^i(z):=\frac12(G^i(z)+\overline{G^i(\overline z)}).
\]
Then $\tilde G^i$ is a holomorphic stem function on $D$ such that $\dd{\tilde G^i}{z}=F^i$. The slice function $g=\I(\sum_{i=0}^3\tilde G^ie_i)$ is a slice-regular primitive of $f$. 

If $\OO=\OO_D$ is a product domain, let $D^+=D\cap\cc^+$. Since $D^+$ is simply connected, there exist holomorphic primitives $G^i_+:D^+\to\cc$ of $F^i$, for $i=0,1,2,3$. Define $G^i_-$ on $D^-:=D\cap\cc^-$ by $G_-^i(z):=\overline{G_+^i(\overline z)}$. Then the function $G^i$ defined as $G^i_+$ on $D^+$ and as  $G^i_-$ on $D^-$ is a holomorphic stem function on $D$ such that  $\dd{G^i}{z}=F^i$. We conclude observing that the sum $\sum_{i=0}^3 G^ie_i$ induces a slice-regular primitive of $f$. 

To prove uniqueness of $g_1,g_2$, we use the linearity of the mapping $(g_1,g_2)\mapsto f$ and Lemma \ref{lem:am}. If $f\equiv0$, then $0\equiv g_1-\overline xg_2$. This means that $g_2$ and $\overline x g_2=g_1$ are axially monogenic, and then $g_2$ (and also $g_1$) is identically zero. 

The last statement is immediate from uniqueness of $g_1$ and $g_2$. If $f$ is slice-preserving, then also $g$ and then $\Delta g$ and $\Delta(xg)$ are slice-preserving. Conversely, if $g_1, g_2$ are slice-preserving, then $g_1-\overline x g_2$ has the same property.
\end{proof}

\begin{corollary}
Formula \eqref{eq:powers} 
holds also for negative integers $n$ and $x\in\hh\setminus\{0\}$.
\end{corollary}
\begin{proof}
If $n=-1$, both sides of formula \eqref{eq:powers} vanish. Let  $n\in\zz$, $n\le -2$. Since $(n+1)^{-1}x^{n+1}$ is a slice-regular primitive of $x^n$ on $\hh\setminus\{0\}$, Theorem \ref{teo:AMDec} gives the monogenic decomposition
\[
x^n=-\tfrac14\Delta((n+1)^{-1}x^{n+2})-\overline x\left(-\tfrac14\Delta((n+1)^{-1}x^{n+1})\right)
\]
Therefore $(n+1)x^n=-\tfrac14\Delta(x^{n+2})-\overline x(-\tfrac14\Delta(x^{n+1}))=\Pn_n(x)-\overline x\Pn_{n-1}(x)$.
\end{proof}

\begin{example}
Let $f(x)=\exp(x)\in\SR(\hh)$. Then $\exp(x)=-\frac14\Delta(x\exp(x))-\overline x\left(-\frac14\Delta\exp(x)\right)$, with $\Delta(x\exp(x))$ and $\Delta\exp(x)$ axially monogenic on $\hh$.
The function $\Delta\exp(x)=-2e^{x_0}(\sinc(\beta)-I_x(\sinc(\beta))')$, 
where $I_x=\tfrac{\IM(x)}{|\IM(x)|}$, $\beta=|\IM(x)|$, coincides up to a multiplicative constant with the function $\text{EXP}_3(x)$ defined in \cite[Ex.11.34]{GHS} in the more general context of Clifford algebras.

The function $\Log(x)\in\SR(\hh\setminus\{x\in\rr\,|\,x\le0\}$ induced by the complex principal logarithm  (see e.g.\ \cite[Ex.5]{AlmansiH}) is a slice-regular primitive of $x^{-1}$ on $\hh\setminus\{x\in\rr\,|\,x\le0\}$. Then we can write
\[
x^{-1}=-\tfrac14\Delta(x\Log(x))+\tfrac{\overline{x}}4\Delta(\Log(x)),
\]
with $\Delta(x\Log(x))$ and $\Delta(\Log(x))$ axially monogenic on $\hh\setminus\{x\in\rr\,|\,x\le0\}$. The function $\Delta(\Log(x))$ coincides up to a multiplicative constant with the $\partial$-primitive  $L(x)$  of the Cauchy-Fueter kernel $E(x)$ defined in \cite[(5.7)]{Su}. Here $\partial$ is the conjugated Cauchy-Riemann-Fueter operator.
\end{example}

The statement of Theorem \ref{teo:AMDec} has a converse. In the following proposition we give a differential condition on the pair $(g_1,g_2)$ that ensures the slice-regularity of $g_1-\overline xg_2$. 

\begin{proposition}
\label{pro:AMDecInverse}
Let $g_1$ and $g_2$ be two axially monogenic functions on an axially symmetric set $\OO$. 
If $f$ is defined as $f(x)=g_1(x)-\overline x g_2(x)$ for every $x\in\OO$, 
then $f$ is slice-regular on $\OO$ if and only if it holds
\begin{equation}\label{eq:nec}
\dd{g_1}{x^c}=g_2+\overline x\dd {g_2}{x^c}.
\end{equation}
\end{proposition}
\begin{proof}
Clearly $f$ is a slice function on $\OO$. 
From Theorem \ref{thm:harmonicity}(1), $f$ is slice-regular if and only if $\sd f=-\dcf f$. But $\sd f=\sd{(g_1)}-\sd{(\overline xg_2)}$, while $-\dcf f=\dcf(\overline xg_2)=\dd{(\overline xg_2)}{x^c}-\sd{(\overline x g_2)}=g_2+\overline x\dd{g_2}{x^c}-\sd{(\overline x g_2)}$. Since again from Theorem \ref{thm:harmonicity}(1) it holds $\sd{(g_1)}=\dd{g_1}{x^c}$, $f$ is slice-regular if and only if \eqref{eq:nec} holds.
\end{proof}

\section{A local Cauchy-type integral formula}

Let $\OO=\OO_D$ be an axially symmetric open set. Assume that every connected component of $D$ is simply connected.
Given $f\in\SR(\OO)$,  we know from the proof of Theorem \ref{teo:AMDec} that it is possible to find a slice-regular primitive $g\in\SR(\OO)$. Let $\ZH(\OO)$ denote the right $\hh$-module of  \emph{zonal harmonic functions with pole 1} on $\OO$, i.e., the quaternionic harmonic functions $h$ on $\OO$, such that $h\circ T=h$ for every orthogonal transformation $T$ of $\hh\simeq\rr^4$ that fixes 1. 

Using Theorem \ref{thm:harmonicity}(2a), we can define a linear operator
\[
\mathcal S:\SR(\OO)\to \ZH(\OO)
\]
that maps $f$ to the spherical derivative $\sd g=-\dcf g$ of any slice-regular primitive $g$ of $f$. This map is well-defined since if $\dd g x=\dd {\tilde g}x$, with $g,\tilde g$ slice-regular, then $g-\tilde g$ is locally constant, and then $\sd g-\sd{\tilde g}=\dcf(\tilde g-g)=0$.  It holds $\mathcal S(a)=a$ for any constant $a\in\hh$ and $\mathcal S(x^n)=(n+1)^{-1}\Zt_n$ for every $n\in\nn$. 
If $f$ is slice-preserving, then $\mathcal S f$ is real-valued. Moreover, $\mathcal S$ is injective, since
\[
\ker(\mathcal S)=\{f\in\SR(\OO)\;|\; \textstyle f=\dd g x, g\in\SR(\OO)\cap\AM(\OO)\}=\{0\}.
\]

Let $E(x)=\frac1{2\pi^2}\overline x/|x|^4\in\AM(\hh\setminus\{0\})$ be the Cauchy-Fueter kernel and let $Dy$ be the 3-form with quaternionic coefficients defined as in \cite[(2.28)]{Su}:
\[
Dy=dy_1\wedge dy_2\wedge dy_3-idy_0\wedge dy_2\wedge dy_3+j dy_0\wedge dy_1\wedge dy_3-k dy_0\wedge dy_1\wedge dy_2,
\]
where $y=y_0+iy_1+jy_2+ky_3\in\hh$ with $y_0,y_1,y_2,y_3$ real.
For every $x\in\hh$,  $y\in\hh\setminus\rr$, with $x\ne y$, define the two quaternionic 3-forms 
\begin{align*}
K_1(x,y):&=\big(E(y-x)(Dy)y-\overline x E(y-x)Dy\big)(y-\overline y)^{-1},
\\
K_2(x,y):&=\big(\overline x E(y-x)Dy-E(y-x)(Dy)\overline y\big)(y-\overline y)^{-1}.
\end{align*}

We are now in the position to prove a Cauchy-type integral formula for slice-regular functions on $\OO$ where the integration is performed on the boundary of a not necessarily axially symmetric open subset of $\OO$. As a consequence, we are able to prove a local Cauchy-type integral formula for slice-regular functions.

\begin{theorem}\label{thm:Cauchy}
Let $f$ be slice-regular on an axially symmetric open set $\OO$. Assume that $\OO=\OO_D$, and that every connected component of $D$ is simply connected. Let $U\subset\hh$ be a bounded open set with rectifiable boundary and such that $\overline{U}\subset\OO$. 
Then
\begin{equation}\label{eq:Cauchy}
f(x)=\int_{\partial U}\big(K_1(x,y)f(y)+K_2(x,y)\mathcal S f(y)\big)
\end{equation}
for every $x\in U$.
\end{theorem}
\begin{proof}
From Theorem \ref{teo:AMDec} we get the decomposition $f(x)=g_1(x)-\overline x g_2(x)$ with axially monogenic components $g_1= -\tfrac14\Delta(xg)$, $g_2=-\tfrac14\Delta g$, where $g\in\SR(\OO)$ satisfies $\dd g x=f$.  From the Cauchy-Fueter integral formula for monogenic, i.e., Fueter-regular, functions (see \cite{Fueter1934} for the original proof and also \cite{Su}, where the formula was proved in its full generality), we get
\[
f(x)=\int_{\partial U}E(y-x)(Dy) g_1(y)-\overline x \int_{\partial U}E(y-x)(Dy) g_2(y)
\]
for every $x\in U$. Now we transform the two integrals using Lemma \ref{lem:Laplacian}. We obtain
\begin{equation}\label{eq:g1}
g_1=(2\IM(x))^{-1}\left(-\sd{(xg)}+\dd{(xg)}{x}\right)=(2\IM(x))^{-1}\left(-\overline x\sd g+x f\right),
\end{equation}
where we used the equality $\sd{(xg)}=\vs g+x_0\sd g=(\vs g+\IM(x)\sd g)+\overline x\sd g=g+\overline x\sd g$, and
\begin{equation}\label{eq:g2}
g_2=(2\IM(x))^{-1}(-\sd g+f).
\end{equation}
Using \eqref{eq:g1}, \eqref{eq:g2} and $\sd g=\mathcal Sf$, we get 
\begin{align*}
&E(y-x)(Dy) g_1(y)-\overline x E(y-x)(Dy) g_2(y)\\
&=\left(E(y-x)(Dy) (2\IM(y))^{-1}y-\overline x E(y-x)(Dy) (2\IM(y))^{-1}\right)f(y)\\
&\quad-\left(E(y-x)(Dy) (2\IM(y))^{-1}\overline y-\overline x E(y-x)(Dy) (2\IM(y))^{-1}\right)\mathcal S f(y)\\
&=K_1(x,y)f(y)+K_2(x,y)\mathcal S f(y)
\end{align*}
and the integral formula is proved.
\end{proof}

\begin{remark}
The 3-forms $K_1$, $K_2$ are real-analytic for $(x,y)\in\hh\times(\hh\setminus\rr)$ with $x\ne y$. If the closure of $U$ does not intersect the real axis, then the two integrals with kernels $K_1$ and $K_2$ converge also separately and the integral formula \eqref{eq:Cauchy} can be written as a sum of two integrals
\[
f(x)=\int_{\partial U}K_1(x,y)f(y)+\int_{\partial U}K_2(x,y)\mathcal S f(y)\quad\text{for every }x\in U.
\]
The same holds also when the boundary $\partial U$ is sufficiently smooth and the intersection $\partial U\cap\rr\ne\emptyset$ is transversal. 
\end{remark}

\begin{corollary}[Local Cauchy-type integral formula]\label{cor:localCauchy}
Let $f$ be slice-regular on an axially symmetric open set $\OO$. For any point $\tilde x\in\OO$, there exists an axially symmetric open neighbourhood $W\subseteq\OO$ of $\tilde x$ such that for any bounded open set $U$ with rectifiable boundary and $\overline{U}\subset W$, it holds 
\begin{equation}\label{eq:localCauchy}
f(x)=\int_{\partial U}\big(K_1(x,y)f(y)+K_2(x,y)\mathcal S f(y)\big)
\end{equation}
for every $x\in U$.
\end{corollary}
\begin{proof}
If $B\subseteq\OO$ is an open ball centred in $\tilde x$, we can take $W$ as the symmetric completion $\widetilde B=\cup_{x\in B}\s_x$ of $B$, where $\s_x=\alpha+\s\beta$ for $x=\alpha+J\beta\in B$. Then $W=\OO_E$ with $E\subset\cc$ having simply connected components ($E$ is a complex disc, or a pair of disjoint conjugate discs, or the union of two intersecting conjugate discs). Since $f\in\SR(W)$, the thesis follows from Theorem \ref{thm:Cauchy}.
\end{proof}

The proof of the preceding corollary shows that the operator $\mathcal S$ can be defined on $\SR(\OO)$ for every axially symmetric domain $\OO=\OO_D$, without further assumptions on $D$. For every pair of slice-regular primitives $g\in\SR(W)$, $\tilde g\in\SR(\widetilde W)$ of $f$, it holds $\sd g=\sd{\tilde g}$ on the intersection $W\cap\widetilde W$. This common value defines $\mathcal Sf$.

In the case of quaternionic polynomials, from the equality $\mathcal S(x^n)=(n+1)^{-1}\Zt_n$ we obtain a more explicit form of the integral formula:

\begin{corollary}
Let $f=\sum_{n=0}^dx^na_n\in\hh[x]$ be a polynomial. Then, for any bounded open set $U\subset\hh$ with rectifiable boundary, it holds
\[
f(x)=\sum_{n=0}^d\int_{\partial U}\left(K_1(x,y)y^n+K_2(x,y)\tfrac{\Zt_n(y)}{n+1}\right)a_n
\]
for every $x\in U$.
\qed
\end{corollary}

\begin{example}\label{ex1}
Let $f(x)=x$ and let $B\subseteq\hh$ be an open ball. Then $\mathcal Sf(x)=\frac12\Zt_1(x)=x_0$. For every $x\in B$ it holds
\begin{align*}
&\int_{\partial B}\left(K_1(x,y)y+K_2(x,y)y_0\right)=\int_{\partial B}(K_1(x,y)+K_2(x,y))y_0+\tfrac12\int_{\partial B}K_1(x,y)(y-\overline y)\\
&\quad=\int_{\partial B}E(y-x)(Dy)y_0+\tfrac12\int_{\partial B}\left(E(y-x)(Dy)y-\overline x E(y-x)Dy\right)\\
&\quad=\int_{\partial B}E(y-x)(Dy)(y_0+\tfrac12 y)-\tfrac{\overline x}2\int_{\partial B}E(y-x)Dy=(x_0+\tfrac12 x)-\tfrac{\overline x}2=x,
\end{align*}
since the function $g_1(x)=x_0+\tfrac12 x$ and any constant function are Fueter-regular. Observe that if $x$ is outside the closure of $B$ the integral formula gives a zero value.
\end{example}

\begin{remark}
The integral formula \eqref{eq:Cauchy} has an interpretation which shows its analogy with the classical Cauchy formula in complex analysis. If $F$ is a complex holomorphic function on a neighbourhood $W$ of $\overline D$, and $G$ is a holomorphic primitive of $F$ on $W$, then the Cauchy's integral formula can be written as
\[
F(x)=\int_{\partial D}dG_y(C(x,y)dy)\quad \text{for every $x\in D$},
\]
where $C(x,y)=(2\pi i)^{-1}(y-x)^{-1}$ is the Cauchy kernel and the real differential $dG_y$ of $G$ at $y$ acts on a complex 1-form $a dy$ as multiplication by $G'(y)=F(y)$. 

If $g$ is a slice-regular primitive of $f$ on $\OO$, its real differential at $y\in\OO\cap\cc_I$ is the left $\cc_I$-linear map given by
\[
dg_y=R_{\dd gx(y)}\circ\pi_I+R_{\sd g(y)}\circ\pi_I^\bot,
\]
where $\pi_I:\hh\rightarrow\hh$ denotes the orthogonal projection onto the real vector subspace $\cc_I$, $\pi_I^\bot=\mathit{id}_\HH-\pi_I$ and $R_a$ is the operator of right multiplication by $a\in\hh$ (see \cite[\S3]{GS2020geometric} and \cite[Cor.~3.2]{JGEA2021}). 
Let $\widetilde {dg_y}:\hh^2\to\hh$ be the left $\hh$-linear extension of $dg_y$ defined by
\[
\widetilde {dg_y}(v,w):=R_{\dd gx(y)}(v)+R_{\sd g(y)}(w).
\]
Then it holds $dg_y(v)=\widetilde {dg_y}(\pi_I(v),\pi_I^\bot(v))$ for every $v\in T_y\OO\simeq\hh$. 
Since $\dd gx(y)=f(y)$ and $\sd g(y)=\mathcal Sf(y)$, the integral formula \eqref{eq:Cauchy} of Theorem \ref{thm:Cauchy} can then be written as
\[
f(x)=\int_{\partial U}\widetilde{dg_y}(K_1(x,y),K_2(x,y))\quad \text{for every $x\in U$},
\]
where the action of the operator $\widetilde{dg_y}$ is extended linearly to quaternionic 3-forms by making it act on the coefficients of the forms.
\end{remark}

We now deduce a local Cauchy-type integral formula for the slice derivatives of $f\in\SR(\OO)$. Since 
$\frac{\partial^n f}{\partial x^n}=\frac{\partial^n f}{\partial {x_0}^n}$ for every $n\in\nn$
(see e.g.\ \cite{GeSt2007Adv}), these formulas can be obtained by computing the derivatives of the kernels $K_1$, $K_2$ w.r.t.\ $x_0$. Let $\partial_0$ denote the partial derivative w.r.t.\ $x_0$. Define $K_1^{(n)}(x,y):=\partial_0^nK_1(x,y)$ and $K_2^{(n)}(x,y):=\partial_0^nK_2(x,y)$.

\begin{proposition}
It holds
\begin{equation}\label{eq:dE0}
\partial_0^nE(y-x)=\tfrac{n!}{2\pi^2}\Pn_{-n-3}(y-x),
\end{equation}
and
\begin{align}
K_1^{(n)}(x,y)&=\tfrac{n!}{2\pi^2}\Pn_{-n-3}(y-x)(Dy)y(y-\overline y)^{-1}\label{eq:K1n}
\\
\notag
&-\tfrac{n!}{2\pi^2}\left(\Pn_{-n-2}(y-x)+\overline x \Pn_{-n-3}(y-x)\right)Dy(y-\overline y)^{-1},\\
K_2^{(n)}(x,y)&=\tfrac{n!}{2\pi^2}\left(\Pn_{-n-2}(y-x)+\overline x \Pn_{-n-3}(y-x)\right)Dy(y-\overline y)^{-1}\label{eq:K2n}
\\
\notag
&-\tfrac{n!}{2\pi^2}\Pn_{-n-3}(y-x)(Dy)\overline y(y-\overline y)^{-1}.
\end{align}
for every $n\in\nn$. In particular, $K^{(n)}_1(x,y)+K^{(n)}_2(x,y)=\tfrac{n!}{2\pi^2}\Pn_{-n-3}(y-x)Dy=\partial_0^nE(y-x)Dy$.
\end{proposition}
\begin{proof}
Since $E$ is axially monogenic, it holds $\partial_0E=(\cdcf+\dcf)E=\cdcf E=\tfrac1{2\pi^2}\cdcf\Pn_{-3}$. Since $\cdcf\Pn_n=(n+2)\Pn_{n-1}$ for every $n\in\zz$ (see \cite[Remark~27]{EigenvaluePbms}), we get $\partial_0E=-\tfrac1{2\pi^2}\Pn_{-4}$ and then, inductively, 
\[
\partial_0^nE(x)=(-1)^n\tfrac{n!}{2\pi^2}\Pn_{-n-3}(x)\text{\quad for every $n\in\nn$}.
\]
Formula \eqref{eq:dE0} follows immediately. In order to obtain \eqref{eq:K1n} and \eqref{eq:K2n}, we must compute also $\partial_0^n(\overline xE(y-x))$. From $\partial_0(\overline xE(y-x))=E(y-x)+\overline x\, \partial_0E(y-x)$ we get inductively $\partial_0^n(\overline xE(y-x))=n\partial_0^{n-1}E(y-x)+\overline x\partial_0^{n}E(y-x)$. Using \eqref{eq:dE0} we obtain
\[
\partial_0^n(\overline xE(y-x))=\tfrac{n!}{2\pi^2}\left(\Pn_{-n-2}(y-x)+\overline x\, \Pn_{-n-3}(y-x)\right).
\]
From this equality and \eqref{eq:dE0}, we get \eqref{eq:K1n} and \eqref{eq:K2n}.
\end{proof}

\begin{corollary}[Local Cauchy-type formula for slice derivatives]\label{cor:localCauchyDn}
Let $f\in\SR(\OO)$. For any point $\tilde x\in\OO$, there exists an axially symmetric open neighbourhood $W\subseteq\OO$ of $\tilde x$ such that for any bounded open set $U$ with rectifiable boundary and $\overline{U}\subset W$, it holds 
\begin{equation}\label{eq:localCauchyDerivatives}
\frac{\partial^n f}{\partial x^n}(x)=\int_{\partial U}\big(K_1^{(n)}(x,y)f(y)+K_2^{(n)}(x,y)\mathcal S f(y)\big)
\end{equation}
for every $x\in U$, $n\in\nn$.
\qed
\end{corollary}

\begin{example}\label{ex2}
Let $f$ and $B\subseteq\hh$ be as in Example \ref{ex1}. For every $x\in B$ it holds
\begin{align*}
&\int_{\partial B}\left(K_1^{(1)}(x,y)y+K_2^{(1)}(x,y)y_0\right)=\int_{\partial B}(K_1^{(1)}(x,y)+K_2^{(1)}(x,y))y_0+\tfrac12\int_{\partial B}K_1^{(1)}(x,y)(y-\overline y)\\
&\quad=\int_{\partial B}\partial_0E(y-x)(Dy)y_0+\tfrac12\int_{\partial B}\left(\partial_0E(y-x)(Dy)y-E(y-x)Dy\right)\\
&\quad-\tfrac12\int_{\partial B}\overline x \partial_0E(y-x)Dy\\
&\quad=\int_{\partial B}\partial_0E(y-x)(Dy)(y_0+\tfrac12 y)-\tfrac12\int_{\partial B}E(y-x)Dy-\tfrac{\overline x}2\int_{\partial B}\partial_0E(y-x)Dy\\
&\quad=\tfrac32-\tfrac12=1=\dd f x,
\end{align*}
since the function $x_0+\tfrac12 x$ is Fueter-regular with derivative $3/2$ w.r.t.\ $x_0$. 
\end{example}

\begin{remark}
One could deduce Cauchy-type estimates for the slice derivatives of a slice-regular function from formula \eqref{eq:localCauchyDerivatives}. However, we observe that they were already obtained more easily by means of the two-dimensional Cauchy's formula (see \cite[Prop. 6.8]{GeStoSt2013})
\end{remark}


\begin{thebibliography}{10}

\bibitem{HFT}
S.~Axler, P.~Bourdon, and W.~Ramey.
\newblock {\em Harmonic function theory}, volume 137 of {\em Graduate Texts in
  Mathematics}.
\newblock Springer-Verlag, New York, second edition, 2001.

\bibitem{BDS}
F.~Brackx, R.~Delanghe, and F.~Sommen.
\newblock {\em Clifford analysis}, volume~76 of {\em Research Notes in
  Mathematics}.
\newblock Pitman (Advanced Publishing Program), Boston, MA, 1982.

\bibitem{CoGeSaAGAG}
F.~Colombo, G.~Gentili, and I.~Sabadini.
\newblock A {C}auchy kernel for slice regular functions.
\newblock {\em Ann.\ Global Anal.\ Geom.}, 37:361--378, 2010.

\bibitem{TheInverseFueter}
F.~Colombo, I.~Sabadini, and F.~Sommen.
\newblock The inverse {F}ueter mapping theorem.
\newblock {\em Commun. Pure Appl. Anal.}, 10(4):1165--1181, 2011.

\bibitem{Cullen}
C.~G. Cullen.
\newblock An integral theorem for analytic intrinsic functions on quaternions.
\newblock {\em Duke Math. J.}, 32:139--148, 1965.

\bibitem{DongQian}
B.~Dong and T.~Qian.
\newblock {U}niform generalizations of {F}ueter's theorem.
\newblock {\em Annali di Matematica}, 200:229--251, 2021.

\bibitem{Fueter1934}
R.~Fueter.
\newblock Die {F}unktionentheorie der {D}ifferentialgleichungen {$\Delta u=0$}
  und {$\Delta\Delta u=0$} mit vier reellen {V}ariablen.
\newblock {\em Comment. Math. Helv.}, 7(1):307--330, 1934.

\bibitem{GS2020geometric}
G.~Gentili and C.~Stoppato.
\newblock Geometric function theory over quaternionic slice domains.
\newblock {\em Journal of Mathematical Analysis and Applications},
  495(2):124780, 2021.

\bibitem{GeStoSt2013}
G.~Gentili, C.~Stoppato, and D.~C. Struppa.
\newblock {\em Regular Functions of a Quaternionic Variable}.
\newblock Springer Monographs in Mathematics. Springer, 2013.

\bibitem{GeSt2006CR}
G.~Gentili and D.~C. Struppa.
\newblock A new approach to {C}ullen-regular functions of a quaternionic
  variable.
\newblock {\em C. R. Math. Acad. Sci. Paris}, 342(10):741--744, 2006.

\bibitem{GeSt2007Adv}
G.~Gentili and D.~C. Struppa.
\newblock A new theory of regular functions of a quaternionic variable.
\newblock {\em Adv. Math.}, 216(1):279--301, 2007.

\bibitem{GhPe_Trends}
R.~Ghiloni and A.~Perotti.
\newblock A new approach to slice regularity on real algebras.
\newblock In {\em Hypercomplex analysis and applications}, Trends Math., pages
  109--123. Birkh\"{a}user/Springer Basel AG, Basel, 2011.

\bibitem{GhPe_AIM}
R.~Ghiloni and A.~Perotti.
\newblock Slice regular functions on real alternative algebras.
\newblock {\em Adv. Math.}, 226(2):1662--1691, 2011.

\bibitem{VolumeCauchy}
R.~Ghiloni and A.~Perotti.
\newblock Volume {C}auchy formulas for slice functions on real associative
  *-algebras.
\newblock {\em Complex Var. Elliptic Equ.}, 58(12):1701--1714, 2013.

\bibitem{Gh_Pe_GlobDiff}
R.~Ghiloni and A.~Perotti.
\newblock Global differential equations for slice regular functions.
\newblock {\em Math. Nachr.}, 287, 2014.

\bibitem{JGEA2021}
R.~Ghiloni and A.~Perotti.
\newblock On a class of orientation-preserving maps of {$\Bbb R^4$}.
\newblock {\em J. Geom. Anal.}, 31(3):2383--2415, 2021.

\bibitem{AlgebraSliceFunctions}
R.~Ghiloni, A.~Perotti, and C.~Stoppato.
\newblock The algebra of slice functions.
\newblock {\em Trans. Amer. Math. Soc.}, 369(7):4725--4762, 2017.

\bibitem{GHS}
K.~G{\"u}rlebeck, K.~Habetha, and W.~Spr{\"o}{\ss}ig.
\newblock {\em Holomorphic functions in the plane and {$n$}-dimensional space}.
\newblock Birkh\"auser Verlag, Basel, 2008.

\bibitem{EigenvaluePbms}
R.~S. Krausshar and A.~Perotti.
\newblock {E}igenvalue problems for slice functions.
\newblock 2021.
\newblock \url{http://arxiv.org/abs/2103.14868}, submitted.

\bibitem{Harmonicity}
A.~Perotti.
\newblock {S}lice regularity and harmonicity on {C}lifford algebras.
\newblock In {\em {T}opics in {C}lifford {A}nalysis -- {S}pecial {V}olume in
  {H}onor of {W}olfgang {S}pr\"o\ss ig}, Trends Math. Springer, Basel, 2019.

\bibitem{AlmansiH}
A.~Perotti.
\newblock {A}lmansi theorem and mean value formula for quaternionic
  slice-regular functions.
\newblock {\em Adv.\ Appl.\ Clifford Algebras}, 30(61), 2020.

\bibitem{Su}
A.~Sudbery.
\newblock Quaternionic analysis.
\newblock {\em Math. Proc. Cambridge Philos. Soc.}, 85(2):199--224, 1979.

\end{thebibliography}

\end{document}